\documentclass{article}

\usepackage{graphicx}
\usepackage[affil-it]{authblk}
\usepackage{tikz}
\usetikzlibrary{angles,quotes}
\usetikzlibrary{arrows, arrows.meta}
\usepackage{amssymb}
\usepackage{amsmath}
\usepackage{amsthm}
\usepackage{hyperref}
\usepackage{biblatex}
\usepackage{enumitem}
\usepackage{geometry}
\bibliography{references} 

\newtheorem{thm}{Theorem}[section]

\newtheorem{lemma}[thm]{Lemma}

\newtheorem{propn}[thm]{Proposition}
\theoremstyle{definition}

\theoremstyle{remark}
\newtheorem{remark}[thm]{Remark}

\title{Corner Regularity of Laplacian Eigenfunctions}
\author{Finlay Whitton}
\affil{Heriot-Watt University \& Maxwell Institute Graduate School}
\date{\today}
\begin{document}

\newcommand{\R}{\mathbb{R}}
\newcommand{\N}{\mathbb{N}}
\newcommand{\Z}{\mathbb{Z}}
\newcommand{\C}{\mathbb{C}}
\renewcommand{\d}{\partial}
\newcommand{\z}{\bar{z}}
\renewcommand{\O}{\Omega}
\renewcommand{\a}{\alpha}
\renewcommand{\b}{\beta}
\renewcommand{\th}{\theta}
\renewcommand{\l}{\lambda}
\renewcommand{\o}{\omega}
\renewcommand{\L}{\Lambda}
\newcommand{\eps}{\varepsilon}
\newcommand{\nn}{\nonumber}

\maketitle

\begin{abstract}
    We study the regularity of eigenfunctions of the Dirichlet Laplacian on domains with corners. The specific question we consider is whether the eigenfunctions can be smooth, e.g. trigonometric, on the closure of their domain. We determine how this regularity depends on the local geometry of the corner. Local asymptotic analysis provides detailed information about eigenfunctions near corners and allows us to quantify the H\"older regularity of the normal derivative. We demonstrate the limitations of some of the methods used to study these problems in finding a complete set of eigenvalues on polygons.
\end{abstract}

\section{Introduction}
The Dirichlet Laplacian operator on a domain $\O \subset \R^2$ is determined by its eigenvalues and eigenfunctions, namely pairs $(\l,u)\in\R\times C^2(\O)$ such that the equation
\begin{equation}
\label{eqn:evaluegen}
    \begin{cases}
        -\Delta u(x,y) = \l u(x,y) &(x,y)\in\O\\
        u(x,y) = 0&(x,y)\in\d\O,
    \end{cases}
\end{equation}
where $\Delta = \d_x^2 + \d_y^2$, admits a non-trivial solution. The spectral theorem implies that the eigenvalues may be arranged sequentially:
$$0<\l_1<\l_2\le \l_3 \le \cdots \nearrow \infty,$$ with $(\l_1,u_1)$ called the ground eigenpair. Despite the definition of a Laplacian eigenfunction requiring only two derivatives, all eigenfunctions belong to $C^\infty(\O)$ for any open, bounded domain $\O\subset\R^2$, with derivatives of all orders extending continuously to $C^\infty$ sections of the boundary (see, for example, \cite[Theorem 2.2.1,][]{TSG}). In this paper  we determine the precise number of derivatives that can be continuously extended to the prototypical example of a non-differentiable boundary point: a corner. 

Consequently, we prove that there are only four polygons whose ground eigenfunctions are infinitely differentiable on the closure of the domain\footnote{We define infinite differentiability on the closure of the domain as in Appendix A of \cite{Evans}:
$$C^k(\bar{\O}) = \{u\in C^k(\O) : D^\alpha u\text{ is uniformly continuous on }\O\text{ for all }|\alpha|\le k\},\quad C^\infty(\bar{\O}) = \bigcap_{k=0}^\infty C^k(\bar{\O}),$$ where the multi-index $\a=(\a_1,\dots,\a_n)\in\N_0^n$ defines the derivative $D^\a = \d_{x_1}^{\a_1}\cdots\d_{x_n}^{\a_n}$.}. These polygons are the rectangle and the equilateral, right isosceles, and hemi-equilateral triangles, which are shown in Figure \ref{fig:canonical}. Herein we refer to these polygons as \textit{canonical}.
\begin{figure}[h]
    \centering
    
        \begin{tikzpicture}
            \coordinate (A) at (0,{sqrt(3)});
            \coordinate (B) at (-1,0);
            \coordinate (C) at (1,0);
            \draw[thick] (A) -- (B) -- (C) -- cycle;
            \node at (0,-0.5) {$\left(\frac{\pi}{3},\frac{\pi}{3},\frac{\pi}{3}\right)$};
        \end{tikzpicture}   
    \quad
        \begin{tikzpicture}
            \coordinate (A) at (0,{sqrt(3)});
            \coordinate (B) at (0,0);
            \coordinate (C) at (1,0);
            \draw[thick] (A) -- (B) -- (C) -- cycle;
            \node at (0.5,-0.5) {$\left(\frac{\pi}{2},\frac{\pi}{3},\frac{\pi}{6}\right)$};
        \end{tikzpicture}
    \quad
        \begin{tikzpicture}
            \coordinate (A) at (0,{sqrt(3)});
            \coordinate (B) at ({sqrt(3)},0);
            \coordinate (C) at (0,0);
            \draw[thick] (A) -- (B) -- (C) -- cycle;
            \node at ({sqrt(3)/2},-0.5) {$\left(\frac{\pi}{2},\frac{\pi}{4},\frac{\pi}{4}\right)$};
        \end{tikzpicture}
    \quad
        \begin{tikzpicture}
            \coordinate (A) at (0,0);
            \coordinate (B) at (0,{sqrt(3)});
            \coordinate (C) at (1,{sqrt(3)});
            \coordinate (D) at (1,0);
            \draw[thick] (A) -- (B) -- (C) -- (D) -- cycle;
            \node at (0.5,-0.5) {$\left(\frac{\pi}{2},\frac{\pi}{2},\frac{\pi}{2},\frac\pi2\right)$};
        \end{tikzpicture}
    \caption{The Canonical Polygons}
    \label{fig:canonical}
\end{figure}
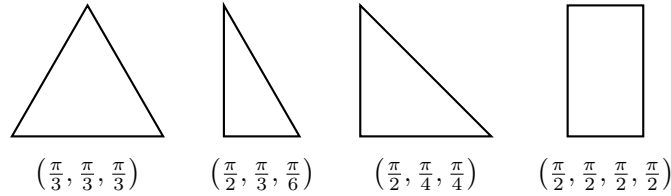
\begin{propn}
    \label{thm:cor}
    The following are equivalent.
    \begin{enumerate}
        \item The polygon $\O$ is canonical.
        \item All Laplacian eigenfunctions on $\O$ are finite sums of plane waves $$\sum_{k=1}^N a_k \cos(m_k x + n_k y) + b_k \sin(p_k x + q_k y) ,$$ with $a_k,b_k,m_k,n_k,p_k,q_k\in\R$.
        \item The ground eigenfunction is smooth on the closure of $\O$, i.e., $u_1\in C^\infty(\overline{\O})$.
    \end{enumerate}
\end{propn}

McCartin \cite[Theorem 4.1.2, ][]{McCartin} showed that \textit{1} and \textit{2} are equivalent. By focusing on the ground eigenfunction we can strengthen this result from `trigonometric' to smooth on the closure.

Away from corners we denote the outward-pointing unit normal vector field to the boundary $\d\O$ by $\nu$. We will precisely determine the regularity of the normal derivative $\d_\nu u = \nu\cdot \nabla u$, quantified by H\"older exponents\footnote{Recall that the space $C^{k,\b}(I)$ contains those functions that are $k$ times continuously differentiable on an interval $I$ and whose $k$th derivative is H\"older continuous with exponent $\b\in(0,1)$. That is, $f\in C^{k,\b}(I)$ if $f\in C^k(I)$ and $$\sup_{x\in I,\\ h>0} \left|\frac{f^{(k)}(x+h)-f^{(k)}(x)}{h^{\b}}\right| <\infty.$$}, from the local geometry. As a result, despite a discontinuity in the normal vector field $\nu$ at the corners, the normal derivative $\d_\nu u$ can be extended to a continuous function along the entire boundary of a convex polygon. For all eigenfunctions we can determine H\"older regularity of the the normal derivative. By restricting attention to the ground eigenfunction, we can conclude that the normal derivative belongs to no H\"older space of more regular functions. Our main result is the following.
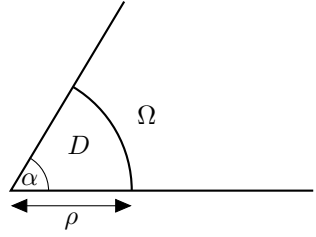
\begin{figure}[h]
    \centering
    \begin{tikzpicture}
        \coordinate (A) at (0,0);
        \coordinate (B) at (4,0);
        \coordinate (C) at (1.5,2.5);
        \draw[thick] (C) -- (A) -- (B);
        \node at (1.8,1) {$\Omega$};
        \pgfmathsetmacro{\angleAC}{atan2(2.5,1.5)}
        \draw[thick] (A) ++(0:1.6) arc[start angle=0, end angle=\angleAC, radius=1.6];
        \draw[>=triangle 45, <->] (0,-0.2) -- (1.6,-0.2);
        \node at (0.8, -0.4) {$\rho$};
        \draw[thin] (A) ++(0:0.5) arc[start angle=0, end angle=60, radius=0.5];
        \node at (0.25,0.15) {$\a$};
        \node at (0.9,0.6) {$D$};
    \end{tikzpicture}
    \caption{$D\subset\Omega$}
    \label{fig:sector}
\end{figure}
Suppose $\O$ has a corner\footnote{A corner is a boundary point at which there exists a ball of positive radius $B_r$ such that $\d\O\cap B_r$ is the union of two straight line segments.} of angle $\a\in(0,\pi)\cup(\pi,2\pi)$. Let $\d_\nu u(r,0)$ and $\d_\nu u(r,\a)$ be the normal derivative of $u$ along the edges, where $r$ measures the distance from the corner.
\begin{thm}
    \label{thm:thm}
     For both $\th=0$ and $\th=\a$, there exists some $\rho >0$ such that the following holds.
    \begin{enumerate}
        \item For $\a\in(0,\pi)$, if $\pi/\a\in\N$ then $\d_\nu u(\cdot,\th) \in C^\infty([0,\rho))$.
        \item Otherwise, 
        $$\d_\nu u(\cdot,\th)\in C^{k(\a),\b(\a)}([0,\rho)), \text{ where }k(\a) = \lfloor\pi/\a\rfloor-1 \text{ and }\b(\a)= \pi/\a - \lfloor\pi/\a\rfloor\in(0,1).$$
        \item For the first eigenfunction, the choices of $k(\a)$ and $\beta(\a)$ are sharp, in the sense that $\d_\nu u_1(\cdot,\th) \notin C^{\tilde{k},\tilde{\b}}([0,\rho))$ for any $\tilde{k}>k(\a)$ or $\tilde{\b}>\b(\a)$. Otherwise, if $\a\in(\pi,2\pi)$ then the normal derivative of the first eigenfunction cannot be extended to a continuous function. That is,
        $$\lim_{r\to 0} |\d_\nu u_1(r,\th)|=\infty.$$
    \end{enumerate}
\end{thm}

Four approaches have been proposed in the literature for studying the eigenvalue  problem~\eqref{eqn:evaluegen} on polygonal domains. We now discuss these in some detail.
\begin{itemize}
    \item \textbf{Algebraic methods}: One can derive the eigenstructure of domains that tile the plane using reflection groups. In 1833 Lam\'e \cite{Lame} was the first to give the eigenstructure of the equilateral triangle, using this method. While he neglected to address the completeness of the given eigenfunctions, Pinsky \cite{Pinsky} established this by combining a reflection argument with the known completeness of the eigenfunctions on the rectangle. The use of reflection principles is demonstrated by B\'erard and Helffer in \cite{BerHelf}. Using their method, one can construct the eigenfunctions for all the canonical polygons. While McCartin \cite{McCartin} used such algebraic arguments to prove that the canonical polygons are the only polygons whose eigenfunctions are trigonometric, our analytic approach below improves that result by proving that these are the only polygons whose respective ground eigenfunctions are smooth on the closure of their domain.
    
    \item \textbf{Separation of variables}: In 2006 McCartin \cite{McCartin} applied separation of variables to a domain that is not obviously separable: the equilateral triangle. Using a barycentric coordinate system, he explicitly found the eigenvalues and eigenfunctions, recovering Lam\'e's result with elementary methods. While the coordinate system can be adapted to arbitrary triangles, the eigenvalue equation splits into a pair of ODEs whose solutions are infinitely differentiable on all of $\R$. Therefore, this method cannot be generalised to find all eigenvalues of an arbitrary triangle since, as stated in Proposition~\ref{thm:cor}, the eigenfunctions of generic triangles are not infinitely differentiable at the corners.

    \item \textbf{Local asymptotic analysis}: Since the seminal work of Lehman \cite{Lehman} in 1959, this approach has mainly received attention in numerical computations, for example, when computing the ground eigenvalue of domains with fractal boundary \cite{BanjBoul} or studying scattering of acoustic waves with large wavenumber off polygons \cite{CWL}. The method exploits the similarity of the local geometry of a polygon with a sector, one of the few domains for which we have explicit expressions for the Dirichlet Laplacian eigenfunctions. This will be the main tool in our proof of Theorem~\ref{thm:thm}, drawing on both real and complex analysis techniques to argue that certain features of sector eigenfunctions are present in the eigenfunctions for polygonal domains at corners. While this technique has the strength of generality, applying to corners of arbitrary angle, it fails to provide any global information. 

    \item \textbf{Unified transform method}: This is the most recent approach to the explicit calculation of eigenvalues. The key feature of this method is the `global relation', which is a functional equation in a complex variable, coupling the known and unknown boundary values of the solution. The unknown boundary values can sometimes be eliminated using either symmetries of the global relation or an ansatz. In 2013 Fokas and Kalimeris \cite{FokKal} used this method to reproduce Lam\'e's solution for the eigenvalues of the equilateral triangle, by making an informed ansatz for the unknown boundary values. 
\end{itemize}
\subsection*{Novelty  and Open Questions}
This paper has its origins in an attempt to extend Fokas and Kalimeris' \cite{FokKal} approach to other polygonal domains. In doing so, we are led to compare the strengths and weaknesses of the four methods above. 

Below we provide an explicit expression for the coefficient of the leading order behaviour of eigenfunctions near a corner and argue when it is non-vanishing, which allows us to derive sharp regularity estimates. While the asymptotic behaviour established by Lehman is well known, the H\"older regularity of Laplacian eigenfunctions does not appear to have been rigorously addressed in previous works. Our main result specifies the regularity of the image of the Dirichlet-to-Neumann map for the problem~\eqref{eqn:evaluegen}.

To adapt the unified transform method to non-canonical polygons, a suitably chosen ansatz must be compatible with the results of Theorem~\ref{thm:thm}. For example, an ansatz for the normal derivative on a regular hexagon must satisfy $\d_\nu u_1 = O(\sqrt{r})$ near each corner.

The existence of an appropriate ansatz for the general form of the normal derivative on a given polygon is still an open problem. Other open problems include the inverse regularity question: seeking polygons whose eigenfunctions have a prescribed degree of regularity. Furthermore, the last statement of Theorem~\ref{thm:thm} relate to the ground eigenfunction but we suspect that there are higher eigenfunctions to which these conclusions apply. Lastly, the radius of convergence of our expansions might be bigger than we realise. In \cite{DescToll}, Descloux and Tolley claim that the series expansions for eigenfunctions near a corner is valid all the way up to the `nearest singularity of $u$', though the paper does not seem to include a proof of this claim.

\subsection*{Structure of the Paper}
The remainder of the paper is organised as follows. In Section~\ref{sect:local} we present the analysis to determine the degree of differentiability of the normal derivative at the corners, using a local expansion of the eigenfunction near the corners in terms of Bessel functions. This section furnishes the proofs of our main results. In Section~\ref{sect:UTM} we describe the steps of the unified transform method and address the scope and limitations of the method in solving the eigenvalue problem on polygons. Finally, in Section~\ref{sect:unfold}, we explain the role of reflection principles in spectral geometry and use them to show that the last part of Theorem~\ref{thm:thm} does not in general extend to higher eigenfunctions.


\section{Local Analysis}
\label{sect:local}
\begin{figure}[h!]
    \centering
    \begin{tikzpicture}
        \coordinate (A) at (0,0);
        \coordinate (B) at (4,0);
        \coordinate (C) at (1.5,2.5);
        \coordinate (x) at (1.35,1);
        \draw[thick] (A) -- (C) -- (B);
        \node at (1.5,2.7) {$\mathbf{v}$};
        \node at (x) {$\mathbf{x}$};
        \draw[dashed] (C) -- (x);
        \node at (1.55,1.75) {$r$};
        \pic[
            draw,
            angle radius=10mm,
            angle eccentricity=1.3,
            "$\theta$"
        ] {angle = A--C--x};
    \end{tikzpicture}
    \caption{Local polar coordinates centred at $\mathbf{v}.$}
    \label{fig:LPC}
\end{figure}
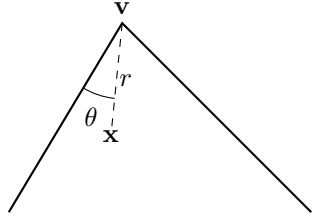

The local geometry of a corner is that of a sector, which is one of the few domains for which we have explicit expressions for the solutions to \eqref{eqn:evaluegen}. This allows us to deduce the local behaviour of eigenfunctions on $\O$. 

Given a corner $\mathbf{v}$ of angle $\a$, we define the local polar coordinates as follows. For a point $\mathbf{x}\in\O$, $r(\mathbf{x})$ is the distance between $\mathbf{v}$ and $\mathbf{x}$, and $\th(\mathbf{x})$ is the positive angle that $\mathbf{x}$ makes when measured anti-clockwise from one of the edges on which $\mathbf{v}$ lies. This is shown in Figure~\ref{fig:LPC}. Our first lemma justifies the notation we use throughout and gives the explicit form of the eigenfunction, summarising the argument given in the proof of \cite[Theorem 3.3,][]{CWL} for the second statement. Let $(\l, u)$ denote an eigenpair on $\O$, a domain with a corner of angle $\a$, and $D\subset\O$ be a sector centred at that corner with radius $0<\rho<\frac{\pi}{2\sqrt{\l}}$, as shown in Figure~\ref{fig:sector}.
\begin{lemma}
     Let $\tilde{u}:(0,\rho)\times(0,\a)\to\R$ denote the restriction of $u$ to $D$ in local polar coordinates. 
    \begin{enumerate}
        \item The function $\tilde{u}$ is given by the series
            \begin{equation}
                \tilde{u}(r,\theta) = \sum_{n = 1}^\infty a_n J_{n\pi/\a}(\sqrt{\l}r)\sin\left(\frac{n\pi\th}{\a}\right)
                \label{eqn:sectorefunc}
            \end{equation}
            with coefficients
            \begin{equation}
                a_n = \frac{2}{\a J_{n\pi/\a}(\sqrt{\l}\rho)}\int_0^\a \tilde{u}(\rho,\phi)\sin\left(\frac{n\pi\phi}{\a}\right)\,d\phi,\quad n\in\N.
                \label{eqn:coeffs}
            \end{equation}
        \item The series~\eqref{eqn:sectorefunc} converges uniformly for $r<\rho$.
    \end{enumerate}
    \label{thm:series}
\end{lemma}

\begin{proof} 

    \begin{enumerate}
        \item Consider the inhomogeneous problem on $D$
    \begin{equation}
        \begin{cases}
            (\Delta +{\l})\varphi (r,\th) = 0 & (r,\th)\in(0,\rho)\times(0,\a)\\
            \varphi(r,0) = \varphi(r,\a) = 0 & r\in (0,\rho)\\
            \varphi(\rho,\th) = G(\th)
            \label{eqn:poisson}
        \end{cases}
    \end{equation}
    where ${\l}$ and $G(\th) = \tilde{u}(\rho,\th)$ are given. Firstly, $\tilde{u}$ is automatically a solution to this problem. Equation~\eqref{eqn:poisson} can also be solved using separation of variables to give
    \begin{equation}
        \psi(r,\th) = \sum_{n = 1}^\infty a_n J_{n\pi/\a}(\sqrt{{\l}}r)\sin\left(\frac{n\pi\th}{\a}\right),
        \label{eqn:series}
    \end{equation}
    as in \cite{BanjBoul}, where $J_m$ is an $m$th order Bessel function of the first kind and
    $$a_n = \frac{2}{\a J_{n\pi/\a}(\sqrt{\l} \rho)}\int_0^\a G(\phi)\sin(n\pi\phi/\a)\,d\phi.$$
    Finally we claim that a solution to this problem is unique. Consider the problem solved by $\tilde{u}-\psi$ where $\psi$ is the series \eqref{eqn:series}:
    \begin{equation*}
        \begin{cases}
            (\Delta + {\l})(\tilde{u}-\psi) = 0&\text{in }D\\
            \tilde{u}-\psi = 0 &\text{on }\d D.
        \end{cases}
    \end{equation*}
    If the operator $\Delta + \l$ is injective then $\tilde{u}_1-\psi$ is identically 0. This occurs precisely when $\l$ is not a Dirichlet Laplacian eigenvalue on $D$. The eigenvalues on $D$ form a discrete subset of $[0,\infty)$ with the only accumulation point at $\infty$. When we scale $\rho$ by $1-\eps$ for $\eps>0$, all eigenvalues on $D$ are scaled by $(1-\eps)^{-2}$. If $\l$ does happen to be an eigenvalue on $D$ then we can choose a smaller, but arbitrarily close, value of $\rho$ and repeat the argument to prove the first statement.
    \item We first bound the coefficients~\eqref{eqn:coeffs} from above:
    \begin{equation*}
        |a_n| = \left| \frac{2}{\a J_{n\pi/\a}(\sqrt{\l} \rho)}\int_0^\a \tilde{u}(\rho,\phi)\sin(n\pi\phi/\a)\,d\phi\right|\leq \frac{2\|\tilde{u} \|_{L^\infty(D)}}{J_{n\pi/\a}(\sqrt{{\l }}\rho)}.
    \end{equation*}
    Consider the integral representation for Bessel functions of order $\b>-1/2$ \cite[Equation 17.3.1, ][]{WhitWats}
    \begin{equation}
        J_\b(z) = \frac{2(z/2)^\b}{\sqrt{\pi}\Gamma(\b+1/2)}\int_0^1 (1-t^2)^{\b-1/2}\cos(zt)\,dt,\, \text{Re}z>0
        \label{eqn:besselintegral}
    \end{equation}
    For arbitrary $\b$ we choose the principal branch of $(z/2)^\b$, which is analytic in $\text{Re} z>0$. However, for $\b\in\N$ the representation is valid for any $z\in\C$. We use $|\cos(zt)| \le e^{|\Im z|}$ for $z\in\C$ and $t\in(0,1)$ to provide an upper bound
    \begin{equation}
        |J_\b(z)| \le \underbrace{\frac{\int_0^1 (1-t^2)^{\b-1/2}\,dt}{\sqrt\pi 2^{\b-1}\Gamma(\b+1/2)}}_{=: C_\b}z^{\b} e^{|\Im z|}.
        \label{eqn:besselupper}
    \end{equation}
    For $x\in[0,\pi/2]$ we can use $\cos(xt) \ge \cos(x)$ for $t\in(0,1)$ to provide the lower bound
    \begin{equation}
        C_\beta\cos(x)x^\b \le J_\b(x).
        \label{eqn:bessellower}
    \end{equation}
    According to \eqref{eqn:besselupper} and \eqref{eqn:bessellower}, each term of the series \eqref{eqn:sectorefunc} is bounded as follows, once we restrict $\sqrt{\l}\rho<\pi/2$:
    \begin{align}
        \left|a_n J_{n\pi/\a}(\sqrt{\l } r)\right | &\leq \frac{2\|\tilde{u} \|_{L^\infty(D)}}{J_{n\pi/\a}(\sqrt{\l }\rho)} J_{n\pi/\a}(\sqrt{\l } r)\nn\\
        &\le \frac{2\|\tilde{u} \|_{L^\infty(D)}C_{n\pi/\a}}{(\sqrt{{\l }}\rho)^{n\pi/\a}\cos(\sqrt{\l }\rho)} \frac{(\sqrt{{\l }}r)^{n\pi/\a}}{C_{n\pi/\a}}\nn\\
        &=  \frac{2\|\tilde{u} \|_{L^\infty(D)}}{\cos(\sqrt{{\l }}\rho)}\left(\frac r\rho\right)^{n\pi/\a}.
        \label{eqn:summandbound}
    \end{align}
    With
    \begin{equation*}
        M_n = \frac{2\|\tilde{u} \|_{L^\infty(D)}}{\cos(\sqrt{{\l }}\rho)}\left(\frac {R^{\pi/\a}}{\rho^{\pi/\a}}\right)^n,
    \end{equation*}
    the Weierstrass M-test implies that the series \eqref{eqn:sectorefunc} converges uniformly for $(r,\th) \in [0,R]\times[0,\a]$ for any $R\leq\frac{\pi}{2\sqrt{{\l }}}$
    as
    \begin{equation*}
        \sum_{n=1}^\infty M_n < \infty.
    \end{equation*}
    \end{enumerate}
\end{proof}

Having established the uniform convergence of the expansion~\eqref{eqn:sectorefunc} in Lemma \ref{thm:series}, we now describe the asymptotic behaviour of $\d_\nu u $. Lehman \cite[Theorem 3.4, ][]{Lehman} showed that, near a corner, the solution to any second order, homogeneous, elliptic PDE with homogeneous Dirichlet boundary condition has the asymptotic behaviour
\begin{equation*}
    \tilde{u}(r,\th) = \gamma r^{\pi/\a}\sin(\pi\th/\a) + o(r^{\pi/\a}) \text{ as } r\to0^+.
\end{equation*}
The following lemma gives an explicit expression for $\gamma$:
\begin{equation*}
    \gamma =  \frac{a_1 }{\Gamma(\pi/\a+1)}\left(\frac{\sqrt{\l }}{2}\right)^{\pi/\a}.
\end{equation*}
\begin{lemma}\label{lemma:2.2}
    Assume the same hypotheses are in Lemma~\ref{thm:series}. Then we have
    \begin{equation}
        \d_\nu \tilde{u} (r,\th) =-\frac{\pi}{\a}\frac{a_1}{\Gamma(\pi/\a+1)}\left(\frac{\sqrt{\l }}{2}\right)^{\pi/\a}r^{\pi/\a-1} + o(r^{\pi/\a-1}),
        \label{eqn:asymptotic}
    \end{equation}
    as $r\to 0^+$.
\end{lemma}

Moreover, \cite[Lemma 4.2, ][]{Lehman} implies that the relations obtained for partial derivatives of arbitrary order by differentiating formally are valid when $\pi/\a\notin\N$. This implies that
\begin{equation}
    \frac{d^k}{dr^k} \d_\nu \tilde{u}(r,\th) = -\gamma \left(\frac{\pi}{\a}\right)_{(k+1)}r ^{\pi/\a-1-k} + o(r^{\pi/\a-1-k}),
    \label{eqn:kderiv}
\end{equation}    
where $$(t)_{(k)} = t(t-1)\cdots(t-k+1)$$ 
is the falling factorial symbol.
\begin{proof}
    At the point $(r,\th)$ we use $\{\bf{e_r},\bf{e_\th}\}$ as a local orthonormal basis for $\R^2$, where the former is the unit vector in the radial direction of $(r,\th)$ and the latter is the orthogonal tangential vector in the anti-clockwise direction. We find the normal derivative by differentiating \eqref{eqn:sectorefunc} term-wise, thanks to Lemma~\ref{thm:series}, and by using that, in this basis, the outward unit normal vector to the edge $\th=0$ is $-\mathbf{e_\theta}$:
    \begin{align}
        \d_\nu \tilde{u} (r,0) &= -\mathbf{e_\th} \cdot\nabla \tilde{u} (r,\th)\big|_{\th=0}\nn\\
        &= -\mathbf{e_\th}\cdot\left(\frac{\d \tilde{u} }{\d r} \mathbf{e_r} + \frac{1}{r}\frac{\d \tilde{u} }{\d\th}\mathbf{e_\th}\right)\Bigg|_{\th=0}\nn\\
        &= -\frac{1}{r} \frac{\d \tilde{u} (r,\th)}{\d\theta}\Bigg|_{\th=0}\nn\\
        &= -\sum_{n = 1}^\infty \frac{a_nn\pi}{\a}\frac{J_{n\pi/\a}(\sqrt{\l }r)}{r}\cos\left(\frac{n\pi\theta}{\a}\right)\Bigg|_{\th=0}\nn\\
        &= -\sum_{n = 1}^\infty \frac{a_nn\pi}{\a}\frac{J_{n\pi/\a}(\sqrt{\l }r)}{r}.
        \label{eqn:normal1}
    \end{align}
    The outward unit normal vector to the edge $\th=\a$ is $\mathbf{e_\theta}$ so
    \begin{equation} 
        \d_\nu \tilde{u} (r,\a) = \mathbf{e_\th} \cdot\nabla \tilde{u} (r,\th)\big|_{\th=\a} = \sum_{n = 1}^\infty(-1)^n \frac{a_nn\pi}{\a}\frac{J_{n\pi/\a}(\sqrt{\l }r)}{r}.
        \label{eqn:normal2}
    \end{equation}
    We use the series~\eqref{eqn:normal1} to characterise the limit equation~\eqref{eqn:asymptotic} by splitting into a sum of limits, the existence of which is justified \textit{a posteriori}:
    \begin{align}
        &\left(\sum_{n=1}^\infty \frac{n\pi a_n}{\a} \frac{J_{n\pi/\a}(\sqrt{\l }r)}{r} - \frac{a_1 \pi}{\a\Gamma(\pi/\a+1)}\left(\frac{\sqrt{\l }}{2}\right)^{\pi/\a}r^{\pi/\a-1}\right)\bigg/r^{\pi/\a-1} \to 0\nn\\
        \iff &\sum_{n=1}^\infty n a_n \frac{J_{n\pi/\a}(\sqrt{\l }r)}{r^{\pi/\a}} -  \frac{a_1}{\Gamma(\pi/\a+1)}\left(\frac{\sqrt{\l }}{2}\right)^{\pi/\a}\to 0\nn\\
        \iff& \sum_{n=2}^\infty n a_n \frac{J_{n\pi/\a}(\sqrt{\l }r)}{r^{\pi/\a}} + a_1\frac{J_{\pi/\a}(\sqrt{\l }r)}{r^{\pi/\a}} -  \frac{a_1}{\Gamma(\pi/\a+1)}\left(\frac{\sqrt{\l }}{2}\right)^{\pi/\a}\to 0\nn\\
        \Longleftarrow&\underbrace{\sum_{n=2}^\infty n a_n \frac{J_{n\pi/\a}(\sqrt{\l }r)}{r^{\pi/\a}} \to 0}_{\textbf{(I)}} \text{ and }\underbrace{\frac{J_{\pi/\a}(\sqrt{\l }r)}{r^{\pi/\a}} - \frac{1}{\Gamma(\pi/\a+1)}\left(\frac{\sqrt{\l }}{2}\right)^{\pi/\a}\to 0}_{\textbf{(II)}}.\nn
    \end{align}
    Claim \textbf{(II)} asserts that the leading order behaviour of the Bessel function comes from the lowest power in its Frobenius expansion and claim $\textbf{(I)}$ precludes any contributions from the tail sum at leading order. We begin with claim \textbf{(I)}, using inequality~\eqref{eqn:summandbound} to bound the summands
    \begin{align*}
        \left|\sum_{n=2}^\infty na_n \frac{J_{n\pi/\a}(\sqrt{\l }r)}{r^{\pi/\a}}\right| 
        \le& \frac{1}{r^{\pi/\a}}\sum_{n=2}^\infty n\left|a_nJ_{n\pi/\a}(\sqrt{\l }r)\right| \nn\\
        \le&\frac{2\|u\|_\infty}{r^{\pi/\a}\cos(\sqrt{\l }\rho)}\sum_{n=2}^\infty n\left(\frac{r^{\pi/\a}}{\rho^{\pi/\a}}\right)^{n}\\
        =& \frac{2\|u\|_\infty}{r^{\pi/\a}\cos(\sqrt{\l }\rho)}\left[\frac{(r/\rho)^{\pi/\a}}{(1-(r/\rho)^{\pi/\a})^2} - (r/\rho)^{\pi/\a}\right]\nn,
    \end{align*}
    which tends to 0 as $r\to 0$. In the final step we used that, for $|t|<1$,
    $$\sum_{n=1}^\infty nt^n = t\sum_{n=1}^\infty \frac{d}{dt}t^n= t\frac{d}{dt}\frac{1}{1-t} = \frac{t}{(1-t)^2}.$$
    
   Claim \textbf{(II)} follows directly from the Frobenius series for Bessel functions:
    \begin{equation*}
        J_{\pi/\a} (r) = \sum_{m=0}^\infty \frac{(-1)^m}{m!\,\Gamma(m+\pi/\a+1)}\Big(\frac{r}{2}\Big)^{2m+\pi/\a}.
    \end{equation*}
    Indeed,
    \begin{align}
        &\left|\frac{J_{\pi/\a}(\sqrt{\l }r)}{r^{\pi/\a}} -  \frac{1}{\Gamma({\pi/\a}+1)}\left(\frac{\sqrt{\l }}{2}\right)^{\pi/\a}\right|\nn\\
        =& \left|\frac{1}{r^{\pi/\a}}\left(\sum_{m=0}^\infty \frac{(-1)^m}{m!\Gamma(m+{\pi/\a}+1)}\left(\frac{\sqrt{\l }r}{2}\right)^{2m+{\pi/\a}} \right) -  \frac{1}{\Gamma({\pi/\a}+1)}\left(\frac{\sqrt{\l }}{2}\right)^{\pi/\a}\right|\nn\\
        =& \left(\frac{\sqrt{\l }}{2}\right)^{\pi/\a}\left|\sum_{m=1}^\infty \frac{(-1)^m}{m!\Gamma(m+{\pi/\a}+1)}\left(\frac{\sqrt{\l }r}{2}\right)^{2m}\right|\nn\\
        \le& \left(\frac{\sqrt{\l }}{2}\right)^{\pi/\a}\sum_{m=1}^\infty \frac{1}{m!}\left(\frac{\l r^2}{4}\right)^{m} =  \left(\frac{\sqrt{\l }}{2}\right)^{\pi/\a}\left(\exp\left(\l r^2/4\right)-1\right).\nn
    \end{align}
    Taking the limit as $r\to0$ proves the claim. An identical argument proves the claim for the edge $\th=\a$.
\end{proof}
To find sharp upper bounds on the regularity of $\d_\nu u$, we require the first coefficient $a_1$ to be non-zero. This is clear for the first eigenfunction. Indeed,
\begin{equation*}
    a_1 =  \frac{2}{\a J_{\pi/\a}(\sqrt{\l_1}\rho)}\int_0^\a \tilde{u}_1(\rho,\phi)\sin\left(\frac{\pi\phi}{\a}\right)\,d\phi.
\end{equation*}
By Courant's nodal domain theorem, $u_1$ has a definite sign on all of $\O$ and so too does the integrand on right-hand side, meaning $a_1\neq 0$. We are now ready to complete the proof of Theorem~\ref{thm:thm}.
\begin{proof}[Proof of Theorem \ref{thm:thm}]
    We first prove the positive statements. We let $\th=0$ and the $\th=\a$  case is analogous.
    \begin{enumerate}
        \item If $\pi/\a\in\N$ then every Bessel function in \eqref{eqn:normal1} has integer order and is thus infinitely differentiable. However, uniform convergence does not necessarily confer differentiability on the limit function. For this to be true we need complex analyticity, hence we consider a complex extension of the series~\eqref{eqn:normal1}. Let $U = \{ z\in\C : |z|<\pi/2\}$ and
        \begin{equation*}
            F(z) = \frac{1}{ z}\sum_{n=1}^\infty na_n J_{n\pi/\a}(z), \quad z\in U,
        \end{equation*}
        so that $\d_\nu \tilde{u}(r,0) = -\frac{\pi}{\a}F(\sqrt{\l}r)$ for $r\in[0,\rho)$. The sequence of functions given by the partial sums,
        $$F_N(z) = \frac1z\sum_{n=1}^N a_n n J_{n\pi/\a}(z),$$
        converge uniformly to $F$ on compact subsets of $U$ by applying inequality~\eqref{eqn:summandbound} and invoking the Weierstrass M-test. As the integral representation~\eqref{eqn:besselintegral} is valid for all $z\in\C$ when $\pi/\a\in\N$, convergence of $F_N$ is uniform on the disc $U$. As Bessel functions of integer order are entire,
        $$\int_\gamma F_N(z)\,dz = 0,$$
        where $\gamma$ is any contour inside $U$. Uniformity of the convergence of $F_N$ to $F$ allows interchanging limits, so
        $$0 = \lim_{N\to \infty} \int_\gamma F_N(z)\,dz = \int_\gamma \lim_{N\to\infty} F_N(z)\,dz = \int_\gamma F(z)\,dz. $$
        As this is true for all contours $\gamma$ inside $U$, Morera's theorem implies that $F$ is analytic on $U$. Hence, $\d_\nu u$ is infinitely differentiable at $r = 0$.
                
        \item Suppose $\a\in(0,\pi)$ and $\pi/\a \notin \N$. 
        As all derivatives of $u$ can be extended to the $C^\infty$ portions of the boundary, the only point that remains to check is $r=0$. By the identity~\eqref{eqn:kderiv}, we see that the first $k$ derivatives where $\pi/\a-1-k >0$ are continuous and take the value $0$ at the point $r=0$. As for the H\"older continuity, we need only consider the limit as $r\to 0$. Let $\b\in(0,1)$ and consider
        \begin{align}
            & \frac{1}{r^\b}\left( \frac{d^{k(\a)}}{dr^{k(\a)}} \d_\nu \tilde{u}(r,\th)  - \frac{d^{k(\a)}}{dr^{k(\a)}} \d_\nu \tilde{u}(0,\th)\right)\nn\\
            =& \frac1{r^{\b}} \left({ -\gamma \left(\frac{\pi}{\a}\right)_{(k(\a)+1)} r^{\pi/\a-1-k(\a)} + o\left(r^{\pi/\a-1-k(\a)}\right)}\right)\nn\\
            =& \left({ -\gamma\left(\frac{\pi}{\a}\right)_{(k(\a)+1)} r^{\pi/\a-1-k(\a)-\b} + o\left(r^{\pi/\a-1-k(\a)-\b}\right)}\right)\nn
        \end{align}
        as $r\to 0$. If $\b$ is such that $\pi/\a-1-k(\a)-\b \ge 0$ then the limit
        $$\lim_{r\to 0}\frac{ \frac{d^{k(\a)}}{dr^{k(\a)}} \d_\nu \tilde{u}(r,\th)  - \frac{d^{k(\a)}}{dr^{k(\a)}} \d_\nu \tilde{u}(0,\th)}{r^\b}$$
        is finite. Thus $\d_\nu u \in C^{k(\a),\b(\a)}$ where $\b(\a) = \pi/\a - 1 - k(\a) = \pi/\a - \lfloor \pi/\a\rfloor$.
    \end{enumerate}
    To prove the negative statements we restrict our attention to the ground eigenfunction.
    \begin{enumerate}[resume]  
        \item Letting $\tilde{k}>k(\a)$ implies that the exponent $\pi/\a -1-\tilde{k}$ in \eqref{eqn:kderiv} is negative, meaning the $\tilde{k}$th derivative of $\d_\nu u_1$ grows without bound as $r\to0$.

        Now let $\tilde{\b}>\b(\a)$. Once again, the limit under consideration has behaviour
        \begin{equation*}
             { -\gamma\left(\frac{\pi}{\a}\right)_{(k(\a)+1)} r^{\pi/\a-1-k(\a)-\tilde{\b}} + o\left(r^{\pi/\a-1-k(\a)-\tilde{\b}}\right)}
        \end{equation*}
        as $r\to 0$. As the exponent is negative again, the limit does not exist, meaning $\d_\nu \tilde{u}_1(r,\th) \notin C^{k(\a),\tilde{\b}}([0,\rho))$.
        
        For the second part, consider a re-entrant corner of angle $\a\in(\pi,2\pi)$. Because $\d_\nu\tilde{u}_1^j = -\gamma\frac{\pi}{\a}r^{\pi/\a-1} + o(r^{\pi/\a-1})$, the exponent is always negative for this range of angles, meaning
        \begin{equation*}
            \lim_{r\to 0^+}|\d_\nu u_1^j(r)| = \infty.
        \end{equation*}
    \end{enumerate}
    This concludes the proof.
\end{proof}

All statements hold for any domain with a corner as our analysis has hitherto been local. By focusing on polygons, we can bridge between the local behaviour of the eigenfunctions and the regularity on the whole domain. 
\begin{proof}[Proof of Proposition~\ref{thm:cor}]
    The eigenfunctions of the canonical polygons are listed in \cite{KuttSig} and they are all readily extended to smooth functions on all of $\R^2$. 

    It remains to show that the ground eigenfunction of a non-canonical polygon is not smooth on $\overline{\O}$. When searching for polygons with smooth eigenfunctions, we focus on convex polygons whose angles are all of the form $\pi/k$ with $k\in\N$; any non-convex polygon must have a singularity in its normal derivative at each re-entrant corner, by virtue of Theorem~\ref{thm:thm}. 
    
    As the sum of angles in a convex $n$-gon is $(n-2)\pi$, we seek $n\in\N$ and $\{k_i\}_{i=1}^n\subset\N$ such that
    \begin{equation}
        \sum_{i=1}^n \frac{\pi}{k_i} = (n-2)\pi.
        \label{eqn:dioph}
    \end{equation}
    Convexity of $\O$ implies that each $k_i \ge2$, meaning
    \begin{equation*}
        (n-2) = \sum_{i=1}^n \frac{1}{k_i}\le \frac{n}{2}\implies n\le 4.
    \end{equation*}
    First consider the case $n=4$. If, say, $k_1>2$ then we must have $k_i<2$ for some index $i\neq 1$, contradicting convexity. Thus equality in \eqref{eqn:dioph} is attained only by $k_i =2$ for $i=1,2,3,4$, corresponding to a rectangle.

    In the case $n=3$ we seek solutions to
    \begin{equation*}
        \frac{1}{k_1} + \frac{1}{k_2} + \frac{1}{k_3} = 1,\quad k_1,k_2,k_3\in\N.
    \end{equation*}
    At least one of the $k_i$ must be less than 4, otherwise the sum will not exceed $3/4$. Therefore we need only consider the cases when at least one of the integers, say $k_1$, is $2$ or $3$. If $k_1=2$,
    $$\frac{1}{k_2} + \frac{1}{k_3} = \frac{1}{2}\implies 2k_2 + 2l_3 = k_2k_3 \implies (k_2-2)(k_3-2) = 4.$$
    The only integer factorisations with $k_2,k_3>0$ are $4\cdot 1$ and $2\cdot 2$ leading to the solutions $(2,3,6)$ and $(2,4,4)$. If $k_1=3$ then
    $$\frac{1}{k_2} + \frac{1}{k_3} = \frac{2}{3} \implies 3k_2 + 3k_3 = 2k_2k_3 \implies (2k_2-3)(2k_3-3) = 9.$$
    The integer factorisations with $k_2,k_3>0$ are $3\cdot 3$ and $9\cdot 1$, corresponding to the solutions $(2,3,6)$ and $(3,3,3).$ Therefore the only solutions correspond to the canonical polygons.

    This exhaustive search means that a non-canonical polygon has at least one angle not of the form $\pi/n$ for $n\in\N$. Theorem~\ref{thm:thm} implies that neither $\d_\nu u_1$ nor $u_1$ is infinitely differentiable at that corner.
\end{proof}


\section{The Unified Transform Method}
\label{sect:UTM}
The unified transform method (UTM), introduced by Fokas in 1997, is a method to derive integral representations of solutions of linear and integrable non-linear PDEs in two independent variables \cite{Fok97}. The method can be summarised in the following steps:
\begin{enumerate}
    \item Write the PDE in divergence form. 
    \item Integrate the divergence form of the PDE over the whole domain and use Green's theorem to restrict this integral to the boundary of the domain. This equation is called the `global relation'.
    \item The global relation couples the known and unknown boundary values. The unknown boundary values are eliminated using either symmetries of the global relation or an ansatz.
\end{enumerate}

For linear initial-boundary value problems of the form
$$q_t + w(-i\d_x) q = 0,$$
where $w$ is a polynomial of arbitrary degree, this method provides an algorithmic way to derive an effective solution, as unknown boundary values can be eliminated by exploiting the discrete symmetries of the polynomial $w$ \cite{FokSpen}. This provided a new unified technique for solving boundary value problems for PDEs of this form of arbitrary order, including cases where a classical transform method is not available.  Another feature of the UTM is that steps 1 and 2 remain the same for different boundary conditions. This is in contrast to other transform methods, where the required transform depends on the boundary condition. For example, when solving the heat equation on $\R^+\times(0,T)$, the sine/cosine transform is needed for the Dirichlet/Neumann problem respectively. Fokas and Spence provide a more detailed overview in \cite{FokSpen} and a pedagogical introduction is given in \cite{Paras}. 
For linear elliptic problems, we generally require complex-analytical tools to eliminate the unknown boundary values. In some special cases, we can do so by exploiting symmetries. For eigenvalue problems such as \eqref{eqn:evaluegen}, the eigenvalue appearing in the equation constitutes an additional unknown; thus, additional information is required. This information is provided by a judicious choice of ansatz for the form of the unknown boundary data. This requirement, as shown in what follows, is a severe limitation to the use of the UTM to characterise all eigenfunctions on non-canonical polygons.

\subsection{The Implementation of the UTM for the Eigenvalue Problem}

It is useful to identify $\O\subset\R^2$ with its image under the map $(x,y)\mapsto x+iy\in\C$ and we will often write $u(x+iy)$ to mean $u(x,y)$. Consider the rewriting of \eqref{eqn:evaluegen} in the complex plane for a function $u:\O\subset\C\to \R$:
\begin{equation}
    \begin{cases}
        \d_{z}\d_{\z}u(z)+\L u(z)= 0 &z\in \O\\
        u(z) = 0&z\in\d \O,
    \end{cases}
\label{eqn:Helmholtz}
\end{equation}
where $\L = \lambda/4$ because $\d_{\z}\d_z = \frac14 \Delta$. 

\textbf{Step 1:}
To write \eqref{eqn:Helmholtz} in divergence form, we multiply by an unspecified function $v$ and move derivatives from $u$ to $v$:
\begin{align}
    2u_{z\z}v + 2\L uv &= (u_zv)_{\z} + (u_{\z}v)_z - u_z v_{\z} - u_{\z}v_z + 2\L uv\nn\\
    &= (u_zv)_{\z} + (u_{\z}v)_z - (uv_{\z})_z - (u v_z)_{\z} + 2uv_{z\z} + 2\L uv\nn\\
    &= (u_{\z}v -uv_{\z})_z + (u_z v - uv_z)_{\z} + 2u(v_{z\z} + \L v).\nn
\end{align}
This is written in divergence form if we can pick a function $v$ such that $v_{z\z} + \L v =0$. The latter is our original problem but unconstrained by boundary conditions. We posit a separable ansatz $v(z,\z) = p(z)q(\z)$ and seek non-constant functions $p$ and $q$ such that
\begin{equation}
    p'(z)q'(\z) + \L p(z)q(\z) = 0\implies \frac{p'(z)}{p(z)} = -\L \frac{q(\z)}{q'(\z)}.
    \label{eqn:sov}
\end{equation}
The left-hand side of \eqref{eqn:sov} is a function of only $z$ while the right-hand side is a function of only $\z$, meaning both are constant. By convention, we choose $p' = -ikp$ and $q' = (\L/ik)q$ where $k\in\C\setminus\{0\}$ is an arbitrary parameter, meaning 
$$v(z,\z) = p(z)q(\z) = \exp\left[-ikz + \frac{\L}{ik}\z\right]$$
and the divergence form of equation \eqref{eqn:Helmholtz} is
\begin{equation}
    (u_{\z}v -uv_{\z})_z + (u_z v - uv_z)_{\z} = \d_{z}\left(\exp\left[-ikz+\frac{\L}{ik}\z\right]\left(u_{\z}-\frac{\L}{ik}u\right)\right) + \d_{\z}\left(\exp\left[-ikz+\frac{\L}{ik}\z\right](u_z+iku)\right) = 0.
    \label{eqn:divform}
\end{equation}
\textbf{Step 2:}
We find the global relation by integrating \eqref{eqn:divform} over $\O$ and applying Green's theorem:
\begin{equation}
    \int_{\d \O} \left\{\left(\exp\left[-ikz+\frac{\L}{ik}\z\right](u_z(z)+iku(z))\right)\,dz- \left(\exp\left[-ikz+\frac{\L}{ik}\z\right]\left(u_{\z}(z)-\frac{\L}{ik}u(z)\right)\right)\,d\z\right\} = 0.
    \label{eqn:GR0}
\end{equation}
\textbf{Step 3:} 
There are two sides to solving the eigenvalue problem with the UTM. The first is to derive the eigenfunction given an eigenvalue. This process is formal in the sense that it \textit{a priori} assumes that a solution exists and this assumption is later justified by showing that the derived expression is indeed a solution. When solving the eigenvalue equation in a convex polygon, Fokas showed \cite[Proposition 1.3, ][]{Fok01} that an expression containing the known and unknown boundary values solves the equation. In \cite[Proposition 3.1, ][]{Fok01}, it is shown that, for general convex polygons, the unknown boundary values can be eliminated to find a solution by solving a Riemann-Hilbert problem. However, for certain symmetric polygons, the unknown boundary values can be eliminated by purely algebraic manipulations using symmetries of the global relation. This is done for the equilateral triangle in \cite{Dass} and the right isosceles triangle in \cite{Fok01}.

As mentioned above, while the method is equipped to solve a PDE with known coefficients, when both a coefficient (the eigenvalue) and the solution function are unknown, we need to provide additional information in the form of an ansatz for the unknown boundary data. This is what Fokas and Kalimeris did in \cite{FokKal} to explicitly find the eigenvalues of the Laplacian posed on the equilateral triangle with a variety of boundary conditions. This result is a special case, as they exploited both an informed ansatz and a high degree of symmetry in the problem.

To implement the steps of the method in the case of a convex polygon, we first fix a parametrization. Let $\{z_1,\cdots,z_n\}\subset\C$ be the corners of a convex $n$-gon labelled clockwise. The $j$th edge joins $z_{j-1}$ to $z_j$ where $z_0 = z_n$. Define the length of the $j$th edge to be $l_j = |z_j - z_{j-1}|$ and the direction $\o_j = (z_j - z_{j-1})/l_j$. We parametrise the edge by arc length as $z^j(s) = z_{j-1} + s\o_j$ for $s\in[0,l_j]$ and denote $u^j(s) = u(z^j(s))$. With this parametrisation, and taking into account the Dirichlet boundary condition, the global relation \eqref{eqn:GR0} reads   
\begin{align}
   0= &\sum_{j=1}^n \int_{z_{j-1}}^{z_j} \exp\left[-ikz + \frac{\L}{ik}\overline{z}\right]u_z(z)\,dz -\exp\left[-ikz + \frac{\L}{ik}\overline{z}\right]u_{\z}(z)\,d\z\nn\\
    =&\sum_{j=1}^n\int_0^{l_j}\left\{ \exp\left[-ikz^j(s) + \frac{\L}{ik}\overline{z^j(s)}\right]u^j_z(s)\,\frac{dz^j(s)}{ds} -\exp\left[-ikz^j(s) + \frac{\L}{ik}\overline{z^j(s)}\right]u^j_{\z}(s)\,\frac{d\z^j(s)}{ds}\right\}ds\nn\\
    =&\sum_{j=1}^n\int_0^{l_j}\exp\left[-ik(z_{j-1} + \o_j s) + \frac{\L}{ik}(\overline{z_{j-1}}+\overline{\o_j}s)\right](u^j_z(s)\o_j - u^j_{\z}(s)\overline{\o_j})\,ds\nn\\
    =&\sum_{j=1}^n\exp\left[-ikz_{j-1} + \frac{\L}{ik}\overline{z_{j-1}}\right]\int_0^{l_j} \exp\left[(-ik\o_j + \frac{\L}{ik}\overline{\o_j}s\right](u^j_z(s)\o_j - u^j_{\z}(s)\overline{\o_j})\,ds.\nn
\end{align}
We denote the outward normal derivative of $u$ to the $j$th edge by $\d_\nu u^j(s) = \nu \cdot \nabla u\big|_{z^j(s)}$ and examine the derivatives of $u$:
\begin{equation*}
    u^j_z\o_j - u^j_{\z}\overline{\o_j} = u^j_z\o_j - \overline{u^j_z\o_j} = 2i\text{Im}(u^j_z\o_j) = 2i\d_\nu u^j.
\end{equation*}
If we define $f(k;z) = -ikz + \frac{\L}{ik}\bar{z}$ then the global relation can be written
\begin{equation}
    \sum_{j=1}^n \exp(f(k;z_{j-1}))\int_0^{l_j}\exp(f(k;\o_j)s)\d_\nu u^j(s)\,ds = 0.
    \label{eqn:GR1}
\end{equation}
This is the central tool in the analysis of the eigenvalue problem: the power of this equation is that holds for all complex values of $k$. We have transformed the analysis from the physical variable $z$ to the `spectral variable' $k$. We can propose an ansatz for the normal derivative, such as a sum of trigonometric functions in the case of the canonical polygons. In this case we can compute the integral explicitly and separate out the resulting terms using the linear independence of the functions $\exp(f(k;\o_j))$. What follows is a set of rational equations in $k$ that, by equating the locations of the poles and the residues at each pole, fix the parameters in the ansatz and the eigenvalues. This method also recovers the eigenvalues on the square. 

We want to use the global relation to establish the continuity of the normal derivative on convex polygons. While this follows from Theorem~\ref{thm:thm}, we include it to demonstrate a style of argument used in the UTM framework. The conclusion relies on properties of the normal derivative guaranteed by the result of Lemma \ref{lemma:2.2}, which are not evident from the UTM. Integrating equation~\eqref{eqn:GR1} by parts, we find
\begin{equation}
    \sum_{j=1}^n\frac{\exp(f(k;z_{j-1}))}{f(k;\o_j)} \left\{\exp(f(k;\o_j)l_j)\d_\nu u^j({l_j})-\d_\nu u^j(0) -\int_0^{l_j}\exp(f(k;\o_j)s)\frac{d\d_\nu u^j(s)}{ds}\,ds\right\} =0,
    \label{eqn:GR2}
\end{equation}
where the endpoints are understood as one-sided limits. As the above equation holds for all values of $k$, we proceed by studying the asymptotic behaviour of the individual terms as $k\to\infty$ along different rays in the complex plane. Firstly as $k\to\infty$,
\begin{equation*}
    \exp(f(k;z)) = e^{-ikz}e^{-i\L \z/k} = e^{-ikz}(1 + O(k^{-1})).
\end{equation*}
As $|k|\to\infty$, if ${\rm Im}(kz) =0 $ then  $e^{-ikz} = O(1)$, otherwise $e^{-ikz}$ will grow or decay exponentially. Therefore the first two bracketed terms of \eqref{eqn:GR2} are $O(1)$ as $k\to\infty$ along a ray such that $k\o_j\in\R$. As $k\to\infty$ we can neglect the $e^{-i\L \z/k}$ term. In order to invoke the Riemann-Lebesgue lemma, we require $\frac{d\d_\nu u^j(s)}{ds}\in L^1((0,l_j))$.
Because of the smoothness of the boundary, we have for any $\eps>0$
$$\int_\eps^{l_j-\eps}\left|\frac{d\d_\nu u^j(s)}{ds}\right|ds < \infty.$$
We focus on the left endpoint but the argument for the right endpoint is identical. One can pick $\eps>0$ sufficiently small that $\frac{d\d_\nu u^j(s)}{ds}$ is non-negative on $(0,\eps)$ as $\frac{d\d_\nu u^j(s)}{ds}$ has polynomial behaviour around 0 by Equation~\eqref{eqn:kderiv}. Thus
$$\int_0^{\eps}\left|\frac{d\d_\nu u^j(s)}{ds}\right|ds = \int_0^{\eps}\frac{d\d_\nu u^j(s)}{ds}ds = \d_\nu u^j(\eps)-\d_\nu u^j(0) < \infty,$$
under the assumption that $r^{\pi/\a-1}\to0$ as $r\to 0$ which is equivalent to the polygon being convex. Thus, by virtue of the Riemann-Lebesgue lemma,
\begin{equation*}
    \int_0^{l_j} \exp(-ik\o_js)\frac{d\d_\nu u^j(s)}{ds} \,ds\to0\text{ as } k\to\infty.
\end{equation*}
Every other term in the sum with $\o_l\neq\o_j$ either grows or decays exponentially. The integral term decays to zero while the terms containing the normal derivative evaluated at the endpoints remain bounded as $k\to\infty$. For the sum to equal zero for all $k$, the bounded terms must sum to zero and the integral terms must sum to zero. Using this and $f(k;z_{j-1})+f(k;\o_{j-1})l_j =f(k;z_{j-1})+f(k;z_j -z_{j-1})  = f(k;z_j)$ we have
\begin{align}
    &\sum_{j=1}^n\left\{ \frac{\exp(f(k;z_j))\d_\nu u^j({l_j})}{f(k;\o_j)}-\frac{\exp(f(k;z_{j-1}))\d_\nu u^j(0)}{f(k;\o_j)}\right\} = 0.\nn
\end{align}
Re-indexing the second term by mapping $j\mapsto j+1$, we find
\begin{align}
\end{align}
where $\d_\nu u^j({l_j})$ is evaluated at the same corner as $\d_\nu u^{j+1}(0)$. Finally, by the linear independence of the functions $\exp(f(k;z_j))$ as functions of $k\in\C$ this sum vanishes only if each of the coefficients does, meaning
\begin{equation*}
    f(k;\o_{j+1})\d_\nu u^j({l_j}) - f(k;\o_j)\d_\nu u^{j+1}(0) =0 \text{ for all } j=0,\dots,n-1
    \label{eqn:coeffvanish}
\end{equation*}
and once again for this to be true for all $k\in\C$ we must have
\begin{equation*}
    \lim_{r\to l_j}\d_\nu u^j(r) = \lim_{r\to 0} \d_\nu u^{j+1}(r)
\end{equation*}
for all $j.$
The breakdown of infinite differentiability at the corners of non-canonical polygons in Proposition~\ref{thm:cor} demonstrates why a non-smooth ansatz is necessary to generalise the UTM to finding explicit eigenvalues of arbitrary polygons. However, all we have to inform such an ansatz is the asymptotic behaviour of the eigenfunction near the corners.
\begin{remark}
    It is well-known that the eigenfunctions of the Dirichlet Laplacian form a complete set in $L^2(\O)$, however Lam\'e did not prove that his given set of eigenfunctions were complete. This was later established by Pinsky \cite{Pinsky}, whose proof uses the known completeness of the eigenfunctions of the rectangle. In \cite{FokKal}, Fokas and Kalimeris indicate that the completeness of the eigenfunctions they derived follows from employing the UTM to solve the heat equation in the interior of an equilateral triangle with Dirichlet boundary conditions and arbitrary initial data $q_0(x,y)$ \cite{HeatEqui}; evaluating the solution at $t=0$ gives the complete spectral decomposition of the arbitrary function $q_0$. However, some regularity of $q_0$ is required for the pointwise evaluation of the solution to be valid.
\end{remark}

\section{Reflection Principles and Higher Eigenfunctions}
\label{sect:unfold}
We include two counter-examples to explain why the third and fourth statements of Theorem~\ref{thm:thm} do not generalise to higher eigenfunctions.

We begin by describing a way of constructing eigenfunctions on `unfolded' domains. Given a domain $\O$ contained entirely in $\{x<0\}$ with a portion of its boundary on the line $x =0$, we can define the `unfolded' domain $\hat{\O}$ to be the union of $\O$ and its reflection in the line $x = 0$\footnote{We stipulate that the unfolded domain contains the line of symmetry.}. For a Dirichlet eigenfunction $u$ with eigenvalue $\l$ on $\O$, we define the function 
\begin{equation*}
    \hat u (x,y) =
    \begin{cases}
        u(x,y) &x<0\\
        -u(-x,y) & x> 0.
    \end{cases}
\end{equation*}
By the invariance of the Laplacian under $x\mapsto -x$, the function $\hat u$ is a Dirichlet eigenfunction on $\hat{\O}$ with eigenvalue $\l$: we refer to $\hat{u}$ as the `unfolded' eigenfunction. For example, Figure~\ref{fig:unfold} shows the unfolding of eigenfunctions $u$ on the hemi-equilateral triangle $\O$ to form eigenfunctions $\hat u$ on the triangle $\hat\O$ with angles $(\pi/6,\pi/6,2\pi/3)$.

This unfolding principle can be used to prove the first statement of Theorem~\ref{thm:thm} for convex domains. If $\pi/\a\in \N$, unfold the domain $\pi/\a$ times to form a straight edge and unfold once again in this straight edge so that the corner is an interior point of the resulting domain, at which the eigenfunction must be $C^\infty$ \cite{KuttSig}. Our local argument is also valid for non-convex domains.

We can unfold any eigenfunction from $\O$ to $\hat\O$. Thus, it is natural to ask: do all eigenfunctions on $\hat\O$ arise from unfolding? Courant's nodal domain theorem gives us the answer: the first eigenfunction is positive on the whole domain and the second eigenfunction, $u_2$, divides its domain into two connected components on which $u_2$ has a sign. Thus an unfolded eigenfunction cannot be the first as it vanishes on the line of symmetry; moreover, the unfolding of the first eigenfunction on $\O$ produces the second eigenfunction on $\hat\O$.

We can classify the eigenfunctions of a domain with $x=0$ as a line of symmetry. Given an eigenfunction $u$, decompose $u (x,y) = \frac{1}{2}(u(x,y)+ u(-x,y)) + \frac12(u(x,y) - u(-x,y))$ into symmetric and anti-symmetric eigenfunctions. Thus we can partition the eigenfunctions of a domain with a line of symmetry into symmetric ($u(-x,y)=u(x,y)$) or anti-symmetric ($u(-x,y)=-u(x,y)$) about the line of symmetry \cite{KuttSig}. By Courant's theorem, the first eigenfunction must be symmetric and the second must be anti-symmetric.

Returning to the example of polygons, Figure~\ref{fig:unfold} shows a non-canonical triangle with infinitely many trigonometric eigenfunctions. We now know that the first eigenfunction, and in fact any symmetric eigenfunction, cannot be produced in this way. Indeed, consider the local expansion of $\hat u$ about the corner of angle $2\pi/3$:
\begin{equation*}
    \tilde{\hat{u}}(r,\th) = \sum_{n=1}^\infty a_n J_{3n/2}(\sqrt{\l}r)\sin(3n\th/2).
\end{equation*}
Only the fractional Bessel functions obstruct infinite differentiability. As the unfolded eigenfunction are antisymmetric about $\th = \pi/3$, the coefficients
\begin{equation*}
    a_n = \frac{3}{\pi J_{3n/2}(\sqrt{\l}\rho)}\int_0^{2\pi/3} \tilde{u}(\rho,\phi)\sin(3n\phi/2)\,d\phi
\end{equation*}
vanish when $n$ is odd, thus the eigenfunction is infinitely differentiable at the corner $2\pi/3$.

This observation implies the following important fact about the applicability of the UTM. An initial attempt to extend the UTM to finding eigenvalues for non-canonical polygons might involve an ansatz of an arbitrary sum of trigonometric functions for the normal derivative. While this might allow the UTM to find the trigonometric eigenvalues on polygons which arise from successive unfolding of canonical polygons, a  significant portion of the spectrum would be omitted. This omission would not become apparent until the corresponding eigenfunctions had been derived and the nodal sets had been calculated, which would reveal that the first eigenvalue was missing.

If we take this non-canonical triangle and once again unfold it along its shorter side we generate a chevron, as in Figure~\ref{fig:unfold}. At the re-entrant corner, the normal derivative of the second eigenfunction tends to zero, rather than infinity. This is a clear demonstration that the second half of Theorem~\ref{thm:thm} does not hold for all eigenvalues. Once again, the symmetry of the domain means that the first coefficient in the local expansion of $\tilde{\hat u}$ vanishes (and so too do all odd coefficients). Information about higher eigenfunctions will hopefully be gained in future work.

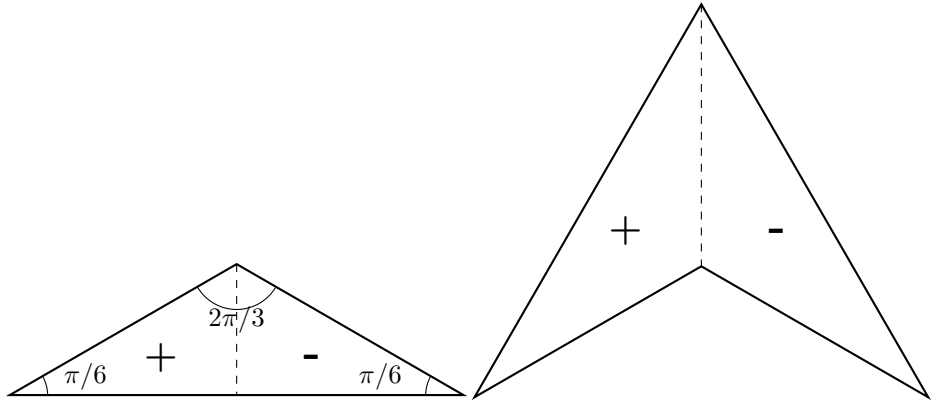
\begin{figure}
    \centering
    \begin{tikzpicture}
        \coordinate (B) at (-3,0);
        \coordinate (C) at (3,0);
        \coordinate (A) at (0,{3*tan(30)});
        \draw[thick] (A) -- (B) -- (C) -- cycle;
        \node at (-1,0.5) {\Large{\textbf{+}}};
        \node at (1,0.5) {\LARGE{\textbf{-}}};
        \draw[dashed] (A) -- (0,0);
        \draw (A) ++(-150:0.6) arc[start angle=-150,end angle=-30,radius=0.6];
        \node at (0,1.0) {$2\pi/3$};
        \draw (B) ++(0:0.5) arc[start angle=0,end angle=30,radius=0.5];
        \node at (-2,0.25) {$\pi/6$};
        \draw (C) ++(180:0.5) arc[start angle=180,end angle=150,radius=0.5];
        \node at (1.9,0.25) {$\pi/6$};
    
    \end{tikzpicture}
    \begin{tikzpicture}
        \coordinate (A) at (-3,0);
        \coordinate (B) at (0,{3*tan(60)});
        \coordinate (C) at (3,0);
        \coordinate (D) at (0,{3*tan(30)});
        \node at (-1,2.2) {\Large{\textbf{+}}};
        \node at (1,2.2) {\LARGE{\textbf{-}}};
        \draw[thick] (A) -- (B) -- (C) -- (D) -- cycle;
        \draw[dashed] (B) -- (D);
    \end{tikzpicture}   
    \caption{The triangle with angles $(2\pi/3,\pi/6,\pi/6)$ and the chevron resulting from its unfolding}
    \label{fig:unfold}
\end{figure}

\textbf{Acknowledgements:} I would like to thank my supervisors Prof. Lyonell Boulton and Prof. Beatrice Pelloni for their continued support and guidance.

\newpage
\printbibliography

\end{document}